\documentclass{amsproc}%
\usepackage{graphicx}
\usepackage{amsmath}
\usepackage{enumerate}
\usepackage{amsrefs}
\usepackage{amsfonts}
\usepackage{amssymb}%
\newtheorem{theorem}{Theorem}[section]

\newtheorem{corollary}[theorem]{Corollary}

\newtheorem{definition}[theorem]{Definition}
\newtheorem{example}[theorem]{Example}

\newtheorem{lemma}[theorem]{Lemma}

\newtheorem{proposition}[theorem]{Proposition}
\theoremstyle{remark}
\newtheorem{remark}[theorem]{Remark}

\begin{document}

\title[Esitmates for heat operators and harmonic functions]{Dimension-independent estimates for heat operators and harmonic functions}
\author{Matthew Cecil}
\address{Department of Mathematics\\Purdue University\\ West Lafayette, IN 47907}
\curraddr{Department of Mathematics\\The University of Utah\\ Salt Lake City, UT 84112}
\email{mcecil{@}math.utah.edu}
\author{Brian C. Hall$^\dagger$}
\address{Department of Mathematics\\University of Notre Dame\\ Notre Dame, IN 46556}
\email{bhall{@}nd.edu}
\thanks{$\dagger$This research was partially supported by NSF Grants DMS-0555862 and DMS-1001328}
\date{\today}
\maketitle

\begin{abstract}
We establish dimension-independent estimates related to heat operators $e^{tL}$ on manifolds.  We first develop a very general
contractivity result for Markov kernels which can be applied to diffusion semigroups.  Second, we develop estimates on the norm behavior of
harmonic and non-negative subharmonic functions.  We apply these results to two examples of interest: when $L$ is the Laplace--Beltrami operator on a Riemannian manifold with  Ricci curvature bounded from below, and when $L$ is an invariant subelliptic operator of H\"{o}rmander type on a Lie group.  In the former example, we also obtain pointwise bounds on harmonic and subharmonic functions, while in the latter example, we obtain pointwise bounds on harmonic functions when a generalized curvature-dimension inequality is satisfied.
\end{abstract}

\section{Introduction and Statement of Results}\label{s.1}

The Laplace--Beltrami operator is of central importance in the study of analysis on a Riemannian manifold $M$. It defines the heat equation and special classes of functions such as harmonic and subharmonic. The Laplace--Beltrami operator can also be used to define heat kernel measures on $M$.  Heat kernel measures continue to exist on many infinite-dimensional manifolds where there is no reasonable notion of volume measure.

We are interested, then, in finding aspects of the study of diffusion operators and heat semigroups on
finite-dimensional manifolds for which the estimates can be made independent
of dimension. Such estimates then have a chance of also holding in the
infinite-dimensional settings.  Some potential infinite-dimensional applications of our results are discussed in the appendix.

The Laplace--Beltrami operator on a Riemannian manifold is the infinitesimal generator of a Brownian motion.  More generally, if $M$ is a smooth manifold, any $M$-valued diffusion $\{X_t\}_{t\geq 0}$ has an infinitesimal generator $L$ and an associated collection of \textit{heat kernel measures} $\mu^x_t$ (depending on $x\in M$) on $\mathcal{B}(M)$, the Borel $\sigma$-algebra of $M$.  The heat kernel measure is defined by
\[\mu^x_t(E)=P_x(X_t\in E)\]
for $E\in \mathcal{B}(M)$, where $P_x$ denotes the probability conditioned on $X_0=x$ a.s.  In the event that $\mu^x_t$ is absolutely continuous with respect to a natural volume measure (for example, Riemannian volume measure or Haar measure), its density is the \textit{heat kernel}, denoted by $\mu_t(x,y)$ in the sequel.  We will assume throughout that $M$ is stochastically complete, i.e. $P_x(X_t\in M)=1$ for all $x\in M$ and all $t\geq 0$.

We define the \textit{heat operator} acting on a function $f:M\rightarrow\mathbb{R}$ by the integral
\begin{equation}\label{e.888}
(e^{tL}f)(x):=\int_M f(y)\mu_t^x(dy)=\mathbb{E}_x[f(X_t)]
\end{equation}
whenever the integral converges.  When the heat kernel is the minimal fundamental solution to the heat equation, the function $u(t,x)=(e^{tL}f)(x)$ solves the heat equation
\[
\frac{\partial}{\partial t}u(t,x)=L u(t,x)\qquad \lim_{t\rightarrow 0^+}u(t,x)=f(x).
\]%

In this article, we discuss three types of dimension-independent estimates, all
of which concern functions that belong to $L^{p}$ with respect to a fixed heat
kernel measure $\mu_T^o$. The first concerns the heat operator itself, the second
gives a certain norm inequality for harmonic and non-negative subharmonic functions, and
the third gives pointwise bounds on harmonic functions in two examples of interest: when $L$ is the Laplace--Beltrami operator on a Riemannian manifold with Ricci curvature bounded from below and when $L$ is an invariant subelliptic operator of H\"{o}rmander type on a Lie group which satisfies a generalized curvature-dimension inequality.

Under mild assumptions (for example, if our heat kernel has a jointly measurable density with respect to a reference measure), the map $x\rightarrow\mu_t^x(E)$ is measurable for every $E\in\mathcal{B}(M)$.  In this event, our first result is the following:
\begin{theorem}\label{t.1.1}
Suppose $0\leq t<T$ and $1\leq
p\leq\infty .$ If $f\in L^{p}(M,\mu^o_{T}),$ then $%
(e^{tL}f)(x)$ (defined by Eq. (\ref{e.888})) is absolutely convergent for $\mu_{T-t}^o$-almost every $x$ in $M.$
Furthermore, $e^{tL}$ is a contraction from $L^{p}(M,\mu^o_{T}) $ to $L^{p}(M,\mu^o_{T-t}).$
\end{theorem}

If $M$ is a Riemannian manifold with Ricci curvature bounded from below and $L$ is the Laplace--Beltrami operator, or more generally if the heat semigroup satisfies a curvature-dimension $CD(K,\infty)$ inequality, then it has already
been proved in \cite{bakry_bolley_gentil} that $e^{tL}$ is hypercontractive and maps $%
L^{p}(M,\mu^o_{T})$ to $L^{q}(M,\mu^o_{T-t})$ for certain values of $q>p$ (see Corollary \ref{c.387} below). Our
result is weaker in the sense that it only shows contractivity rather than
hypercontractivity, but it makes no additional assumptions on $M$ besides stochastic
completeness and measurability of the heat kernels.

Our proof of Theorem \ref{t.1.1} relies on a very general (and elementary)
contractivity result (Theorem \ref{t.558}) for certain maps between finite measure spaces known as Markov kernels.  As such, Theorem \ref{t.1.1} is \textit{directly applicable} to any
infinite-dimensional manifold $M$ provided a reasonable family of heat kernel measures $\{\mu_t^x\}_{x\in M}$ is known to exist.  That is, there is no need to define the heat operator on cylinder functions and then extend by linearity.

Our next results concern harmonic and subharmonic functions on $M$.  At a formal level, a function satisfying $Lf=0$ should also satisfy $e^{tL}f=f$.  However, in the general setting of Theorem \ref{t.1.1}, the measures $\mu_T^o$ and $\mu_{T-t}^o$ may be mutually singular.  In this event, it does not make sense to say that an element of $L^p(M,\mu_T^o)$ is equal to an element of $L^p(M,\mu_{T-t}^o)$.  Even when the measures $\mu_T^o$ and $\mu_{T-t}^o$ are equivalent, the operator $e^{tL}$ is computed by the integral (\ref{e.888}) and not a power series in the operator $L$.  Theorem \ref{t.337} below gives sufficient conditions on the process $X_t$ and its heat kernel to allow one to make rigorous the implication
\[Lf=0\Rightarrow e^{tL}f=f.\]
Theorem \ref{t.337} is then applied to two examples of particular interest.

For the remainder of the introduction, we summarize the results of the two examples found in Sections \ref{s.3.2} and \ref{s.3.3}.  For our first example, let $M$ denote a complete Riemannian manifold with volume measure $V$.  It is known that the heat kernel measures corresponding to $L=\frac{1}{2}\Delta$, where $\Delta$ denotes the Laplace--Beltrami operator, are absolutely continuous with respect to $V$.  The density $\mu_t(x,y)=\frac{d\mu_t^x(y)}{dV(y)}$ is in $C^\infty(M\times M)$ for all $t>0$ (see, for example, Theorem 7.20 of \cite{grigoryan_heat kernel}).  In this case, Eq. (\ref{e.888}) can be written
\begin{equation}\label{e.559}
(e^{t\Delta/2}f)(x)=\int_M f(y)\mu_t(x,y)V(dy).
\end{equation}

\begin{theorem}\label{t.888}
Let $M$ denote a complete Riemannian manifold with $%
Ric(M)\geq -K$ for some $K\geq 0$.  Fix $o\in M$ and $1<p<\infty$.

\begin{enumerate}
\item If $f\in L^{p}(M,\mu _{T}^{o})$ is harmonic then for all $0\leq t\leq T$ and
all $x\in M,$ we have%
\begin{equation}
\int_{M}f(y)\mu _{t}(x,y)V(dy)=f(x),  \label{harmonicHeat}
\end{equation}%
where the integral on the left-hand side of (\ref{harmonicHeat}) is
absolutely convergent. Thus, in the notation of Eq. (\ref{e.559}), we have%
\[
(e^{t\Delta /2}f)(x)=f(x)
\]%
for all $x\in M.$  

\item If $f\in L^{p}(M,\mu _{T}^{o})$ is $C^{2},$ real-valued, and
subharmonic, then for all $0\leq t\leq T$ and all $x\in M,$ we have%
\begin{equation}
\int_{M}f(y)\mu _{t}(x,y)V(dy)\geq f(x),  \label{subharmHeat}
\end{equation}%
where the integral on the left-hand side of (\ref{subharmHeat}) is
absolutely convergent. Thus, in the notation of Eq. (\ref{e.559}), we have
\[
e^{t\Delta /2}f(x)\geq f(x)
\]%
for all $x\in M.$  
\end{enumerate}
\end{theorem}

Note that both statements are obvious at a formal level. After all, \textit{%
if} we could differentiate under the integral sign and integrate by parts
(neither of which is obvious), we would get%
\[
\frac{d}{dt}\int_{M}f(y)\mu _{t}(x,y)V(dy)=\frac{1}{2}\int_{M}(\Delta
f)(y)\mu _{t}(x,y)V(dy).
\]%
\textit{If} we could justify this assertion, then in the harmonic case we
would have that $(e^{t\Delta /2}f)(x)$ is independent of $t,$ whereas in the
subharmonic case we would have that $(e^{t\Delta /2}f)(x)$ increases with $t.
$ In both cases, we expect that $(e^{t\Delta /2}f)(x)$ tends to $f(x)$ as $%
t\rightarrow 0,$ which would establish (\ref{harmonicHeat}) and (\ref%
{subharmHeat}).

However, since functions in $L^{p}(M,\mu^o_{T})$ can
grow very rapidly at infinity, 
it is
not a simple matter to justify differentiating under the integral sign or
integrating by parts. Although it is possible to use bounds on the heat
kernel and its derivatives to justify these steps, we will instead use a
probabilistic argument---which will, of course, still require some heat kernel
bounds.  However, much of our argument is applicable to settings beyond this well-studied case.
\begin{remark}
It should be noted that Proposition 1.8 of Driver \& Gordina \cite{driver_gordina} proves item $(1)$ of Theorem \ref{t.888} in the case where $M$ is a unimodular Lie group with a left-invariant Riemannian metric.  Indeed, this present work can be viewed in part as a generalization of their results.  In Corollary \ref{c.388} below, we also employ their technique of obtaining pointwise bounds on harmonic functions.  We also note that Lemma 5.2 of Driver \& Gross \cite{driver_gross} also gives item $(1)$ when $p=2$ and $f$ is a holomorphic function on a complex Lie group.
\end{remark}
If $f$ is harmonic or $C^2$ and subharmonic, and if $f\in L^p(M,\mu_T^o)$ for some $1<p<\infty$, then Theorem \ref{t.888} indicates that $|f|\leq |e^{(T-t)\Delta/2}f|$ whenever $0\leq t\leq T$. An application of Theorem \ref{t.1.1} immediately gives
\begin{equation}
||f||_{L^p(M,\mu_{t}^o)}\leq ||e^{(T-t)\Delta/2}f||_{L^p(M,\mu_{t}^o)}\leq ||f||_{L^p(M,\mu_T^o)}.
\end{equation}
In Theorem 3.1 of Bakry, Bolley, \& Gentil \cite{bakry_bolley_gentil}, it is shown that for $0\leq t<T$ and $p>1$,
\[||e^{(T-t)\Delta/2}f||_{L^q(M,\mu_t^o)}\leq ||f||_{L^p(M,\mu_T^o)}\]
where $q$ is given by Eq. (\ref{e.1.8}) below.  Thus, we arrive at the following corollary.

\begin{corollary}
\label{c.387}Suppose $f\in L^{p}(M,\mu _{T}^{o})$ for some $1<
p<\infty .$  If $f$ is harmonic or $C^{2},$ non-negative,
and subharmonic, then
\begin{equation}
\left\Vert f\right\Vert _{L^p(M,\mu_s^o)}\leq \left\Vert f\right\Vert_{L^p(M,\mu_t^o)} \leq \left\Vert
f\right\Vert _{L^p(M,\mu_T^o)}  \label{stTineq}
\end{equation}%
for all $0\leq s\leq t\leq T.$

In addition, using the results of \cite{bakry_bolley_gentil}, we also obtain that for all $0\leq t<T,$
\[
\left\Vert f\right\Vert _{L^{q}(M,\mu _{t}^{o})}\leq \left\Vert f\right\Vert
_{L^{p}(M,\mu _{T}^{o})},
\]%
where $q$ is given by%
\begin{equation}
q=1+(p-1)\frac{1-e^{-KT}}{1-e^{-Kt}}.\label{e.1.8}
\end{equation}
\end{corollary}

Using a Harnack inequality, one can show that for any $0<t<T,$
there is a constant $c_{t,T}$ such that for all $x\in M$%
\[
\mu _{t}(x,o)\leq c_{t,T}\mu _{T}(x,o).
\]%
Such an estimate holds because $\mu _{t}(x,o)$ tends to zero at infinity
more rapidly than $\mu _{T}(x,o).$ Thus, for any $f,$ not necessarily
harmonic or subharmonic, we have%
\begin{equation}
\left\Vert f\right\Vert _{L^p(M,\mu_t^o)}\leq c_{t,T}\left\Vert f\right\Vert _{L^p(M,\mu_T^o)}.
\label{cstIneq}
\end{equation}%
Nevertheless, the constants $c_{t,T}$ are certainly not dimension
independent, and they can even become large for a fixed $M.$

In infinite-dimensional examples where heat kernel measures have been
constructed, it is typically the case that $\mu _{s}^{o}$ and $\mu _{t}^{o}$
are mutually singular whenever $s\neq t.$ In such a setting, there cannot be
any inequality like (\ref{cstIneq}) for general functions. Nevertheless,
since there are no constants in (\ref{stTineq}) depending on the dimension
(or anything else), we may hope that this equality continues to hold in the
infinite-dimensional case, provided we can give the result a suitable
interpretation. We discuss the potential infinite-dimensional applications
of our results in Appendix A.

Corollary \ref{c.387} also allows us to obtain pointwise bounds on harmonic and non-negative
subharmonic functions on a complete Riemannian manifold with Ricci curvature bounded from below using an integrated version of Wang's Harnack inequality \cite{wang_logarithmic,wang_equivalence}, as
derived by Driver \& Gordina (Appendix D of \cite{driver_gordina}).
Specifically, it is shown in \cite{driver_gordina}, that when $Ric\geq -K$, $K\geq 0$, and $f\in L^p(M,\mu_T^o)$, one has pointwise bounds on $|e^{t\Delta/2}f(x)|$ of the form of the right hand side of Eqs. (\ref{e.142}) or (\ref{e.140}) below.  When $f$ is harmonic or $C^2$ non-negative and subharmonic, it is a simple matter to replace $|e^{t\Delta/2}f(x)|$ by $|f(x)|$ using Theorem \ref{t.888} and arrive at the following corollary.  It should be noted that Driver \& Gordina have essentially this result in Corollary 1.9 of \cite{driver_gordina} in the case where $M$ is a unimodular Lie group with a left-invariant Riemannian metric.

\begin{corollary}\label{c.388}
Suppose $M$ is a complete Riemannian manifold with $Ric\geq -K$ for some $K\geq 0$.  Suppose $f$ belongs to $%
L^{p}(M,\mu^o_{T}),$ with $1<p< \infty ,$ and that $f$ is harmonic or $C^2$ and
non-negative subharmonic. Then we have the pointwise bounds%
\begin{equation}\label{e.142}
\left\vert f(x)\right\vert \leq \left\Vert f\right\Vert _{L^{p}(M,\mu^o
_{T})}\exp{\left(\frac{K(p-1)}{2(1-e^{-KT})}d^2(x,o)\right)}
\end{equation}
when $K>0$, and
\begin{equation}\label{e.140}
\left\vert f(x)\right\vert \leq \left\Vert f\right\Vert _{L^{p}(M,\mu^o
_{T})}\exp{\left(\frac{(p-1)}{2T}d^2(x,o)\right)}
\end{equation}
when $K=0$.
\end{corollary}

For our second example, we consider $M=G$, a finite dimensional Lie group.  Suppose $\{Y_1,Y_2,...,Y_d\}\subset \mathfrak{g}=\operatorname{Lie}(G)$ satisfy H\"{o}rmander's condition (see Section \ref{s.3.3} for details).  Let $L$ denote the left-invariant operator
\[L=\frac{1}{2}\sum_{k=1}^d \widetilde{Y_k}^2,\]
where $\widetilde{Y}$ denotes the left-invariant extension of $Y\in \mathfrak{g}$.  Note that $L$ is symmetric with respect to a right-invariant Haar measure $\lambda$.  The following theorem is an application of Theorem \ref{t.337} with details found in Section \ref{s.3.3} below.
\begin{theorem}\label{t.1.4}
Suppose $f\in L^p(G,\mu_T^o)$ for some $1< p<\infty$ and $Lf=0$.  Then for all $0\leq t<T$ we have%
\[
e^{tL}f=f
\]%
and for $0\leq s\leq t<T,$%
\[
\left\Vert f\right\Vert _{L^{p}(G,\mu _{s}^{o})}\leq \left\Vert f\right\Vert
_{L^{p}(G,\mu _{t}^{o})}\leq \left\Vert f\right\Vert _{L^{p}(G,\mu
_{T}^{o})}.
\]
\end{theorem}

In certain cases, the Lie group $G$ and the operator $L$ will satisfy a generalized curvature-dimension inequality.  With the above setup in mind, $G$ is said to satisfy the generalized curvature-dimension inequality $CD(\rho_1,\rho_2,\kappa,d)$ if there exist constants $\rho_1\in \mathbb{R}$, $\rho_2>0$, $\kappa\geq 0$, and $d\geq 1$ such that
\[\Gamma_2(f,f)+\nu\Gamma_2^Z(f,f)\geq \frac{1}{d}(Lf)^2+(\rho_1-\frac{\kappa}{\nu})\Gamma(f,f)+\rho_2\Gamma^Z(f,f)\]
for all $\nu>0$ and $f\in C^\infty(G)$, where
\[\Gamma(f,g)=\sum_{i=1}^d (\widetilde{Y}_if)(\widetilde{Y}_ig),\]
\[\Gamma_2(f,f)=\frac{1}{2}(L\Gamma(f,f)-2\Gamma(f,Lf)),\]
$\Gamma^Z(f,g)$ is a first-order differential bilinear form involving differentiation in the `vertical' directions and satisfying certain assumptions (see Section 1 of \cite{BG1} for more details), and
\[\Gamma_2^Z(f,g)=\frac{1}{2}(L\Gamma^Z(f,f)-2\Gamma^Z(f,Lf)).\]
Curvature-dimension inequalities apply in very general sub-Riemannian settings, however, we will only use the application to the Lie group setting described above. It is shown in Section 2 of \cite{BG1} that the 3-dimensional Lie groups $SU(2)$, $H^1$ (the Heisenberg group), and $SL(2,\mathbb{R})$, endowed with natural subriemannian structures, satisfy the generalized curvature dimension inequality $CD(\rho_1,\frac{1}{2},1,2)$ for $\rho_1=1,0,$ and $-1$ respectively.  Note that if $G$ satisfies the generalized curvature-dimension inequality $CD(\rho_1,\rho_2,\kappa,d)$ for $d<\infty$, then it will also satisfy the generalized curvature-dimension inequality $CD(\rho_1,\rho_2,\kappa,\infty)$.

The notion of a generalized curvature-dimension inequality was introduced by Baudoin \& Garofalo in \cite{BG1}, and such inequalities have been shown to imply a wide variety of heat kernel estimates (see \cite{BG1,BBlog,BB2}).  Of particular relevance is the existence of a dimension-independent Harnack inequality of the type first considered by Wang and described above (see Proposition 3.4 of \cite{BBlog}).

\begin{corollary}\label{c.389}
Suppose $G$ satisfies the generalized curvature-dimension inequality $CD(\rho_1,\rho_2,\kappa,\infty)$ for some $\rho_1\in\mathbb{R}$, $\rho_2>0$, and $\kappa\geq 0$.  Suppose $f$ belongs to $%
L^{p}(G,\mu^o_{T}),$ with $1<p< \infty ,$ and that $Lf=0$.  Then we have the pointwise bounds%
\begin{equation}\label{e.142}
\left\vert f(x)\right\vert \leq \left\Vert f\right\Vert _{L^{p}(M,\mu^o
_{T})}\exp{\left(\frac{p}{p-1}\left(\frac{1+\frac{2\kappa}{\rho_2}+2\rho_1^-t}{4t}d^2(x,o)\right)\right)},
\end{equation}
where $\rho_1^-=\max{(-\rho_1,0)}$.
\end{corollary}

\section{A general contractivity result}

Let $(X,\Omega )$ and $(Y,\Xi )$ be measurable spaces and $\nu _{1}$ a
finite measure on $X.$ Suppose that $\{\mu^x\}_{x\in X}$ is a family of
probability measures on $(Y,\Xi )$ depending measurably on $x\in X,$ i.e. for all $E\in \Xi$ the function $x\rightarrow \mu^x(E)$ is $\Omega$-measurable.  Such a collection of measures is known as a \textit{Markov kernel}.  Further information on Markov kernels can be found in Chapter VIII of \cite{bauer}.  Let $\nu _{2}$ be the measure on $(Y,\Xi )$ defined by%
\begin{equation}\label{e.399}
\nu _{2}(E):=\int_{X}\mu^x(E)\nu _{1}(dx).
\end{equation}
Then $\nu _{2}$ is also a finite measure, with $\nu _{2}(Y)=\nu _{1}(X).$  The main result of this section is the following theorem.

\begin{theorem}\label{t.558}
For any $1\leq p\leq\infty$, the map $A$ given by%
\begin{equation}
(Af)(x)=\int_{Y}f(y)\mu^x(dy)  \label{af}
\end{equation}%
is a well-defined map from $L^{p}(Y,\nu _{2})$ into the space of equivalence
classes of $\nu _{1}$-almost everywhere equal functions on $X$.  This means
that given an equivalence class in $L^{p}(Y,\nu _{2}),$ the
integral in (\ref{af}) converges for $\nu _{1}$-almost every $x,$ and that the value of the integral
is independent of the choice of representative, up to a set of $\nu _{1}$%
-measure zero. Furthermore,
\begin{equation}\label{contractivityp}
||Af||_{L^p(X,\nu_1)}\leq ||f||_{L^p(Y,\nu_2)}.
\end{equation}
Thus, $A$ is a contraction of $L^{p}(Y,\nu _{2})$ into $L^{p}(X,\nu _{1}).$
\end{theorem}

\begin{remark}\label{r.3.3.1}
The measures $\mu^x$ do not in general have to be
absolutely continuous with respect to $\nu _{2}.$ It is even possible for $\mu^x$ to be singular with respect to $\nu _{2}$ for \textit{%
every} $x$ in $X.$ Thus, for a fixed $x,$ two functions equal $\nu _{2}$%
-almost everywhere may not be equal $\mu^x$-almost everywhere. As a
result, for a fixed $x,$ the integral in (\ref{af}) may not be independent
of the choice of representative. However, if $f_{1}$
and $f_{2}$ are equal $\nu _{2}$-almost everywhere, then $%
\int_{Y}f_{1}(y)\mu^x(dy)$ agrees with $\int_{Y}f_{2}(y)\mu^x(dy)$
\textit{for }$\nu _{1}$\textit{-almost every }$x.$
\end{remark}

\begin{remark}\label{r.3.3.8}
In the event that there exists a reference measure $\lambda$ on $(Y,\Xi)$ and a jointly measureable non-negative density $m$ on $X\times Y$ such that
\[\mu^x(E)=\int_E m(x,y)\lambda(dy),\]
for all $E\in\Xi$, then $x\rightarrow \mu^x(E)$ is $\Omega$-measurable for all $E\in \Xi$ by Tonelli's Theorem.  
\end{remark}

\begin{proof}
If $E\in \Xi$ and $f(y)=1_{E}(y)$, then $(Af)(x)=\mu^x(E)$ is measurable by our assumption on the collection of measures $\{\mu^x\}_{x\in X}$.  Note as well that
\[
\int_X Af(x)\nu_1(dx)=\int_X \mu^x(E)\nu_1(dx)=\nu_2(E)=\int_Y f(y)\nu_2(dy).
\]
Using the linearity of $A$ and the integral, the equality
\begin{equation}\label{e.001}
\int_X Af(x)\nu_1(dx)=\int_Y f(y)\nu_2(dy)
\end{equation}
holds for any simple function $f$ as well.

Now suppose $f$ is a measurable function on $Y$ with values in $[0,+\infty]$.  Let $\{f_n\}_{n=1}^\infty$ be a sequence of non-negative simple functions which increase monotonically to $f$ everywhere.  Note that $\{Af_n\}_{n=1}^\infty$ is a monotone sequence of non-negative measurable functions on $X$.  By monotone convergence, for any $x\in X$, $(Af)(x)=\lim_{n\rightarrow \infty}(Af_n)(x)$.  So $Af$ is measurable since it is the pointwise limit of measurable functions.  Furthermore, using monotone convergence twice and the result above for simple functions,
\begin{align*}\int_X Af(x)\nu_1(dx)&=\int_X\left(\lim_{n\rightarrow \infty}(Af_n)(x)\right)\nu_1(dx)\\&=\lim_{n\rightarrow\infty}\int_X (Af_n)(x)\nu_1(dx)\\&=\lim_{n\rightarrow\infty}\int_Y f_n(y)\nu_2(dy)\\&=\int_Y f(y)\nu_2(dy).
\end{align*}
So (\ref{e.001}) holds for all non-negative measurable functions $f$, where we allow both sides to be $+\infty$.

Suppose now that $f$ is a measurable complex-valued function on $Y$ with the property that
\begin{equation}
\int_{Y}\left\vert f(y)\right\vert ^{p}\nu _{2}(dy)<+\infty
\label{pfinite2}
\end{equation}%
for some $1\leq p<\infty$.  For fixed $x\in X$, Jensen's inequality gives
\begin{equation}
\left(A|f|(x)\right)^p=\left(\int_Y |f|\mu^x(dy)\right)^p\leq \int_Y |f|^p\mu^x(dy)= A|f|^p(x)\label{e.92370}
\end{equation}
(Although Jensen's inequality is usually stated for integrable functions, if $f \in L^p(Y,\mu^x)$, then $f \in L^1(Y,\mu^x)$. Thus, (\ref{e.92370}) is always valid, although both sides of the inequality may have the value $+\infty$).

Eq. (\ref{e.92370}) and (\ref{e.001}) then give
\begin{equation}
\int_X\left(A|f|(x)\right)^p \nu_1(dx)\leq \int_X A|f|^p(x)\nu_1(dx)=\int_Y |f(y)|^p\nu_2(dy)<+\infty,\label{e.92371}
\end{equation}
which implies that $A|f|(x)<\infty$ for $\nu _{1}$-almost every $x.$  Thus, the quantity $Af(x)$, as defined by (\ref{af}), is well defined (absolutely convergent) for $\nu _{1}$-almost every $x$.  For all $x$ such that $A|f|(x)<\infty$, we can write $Af(x)=Af_1(x)-Af_2(x)+i(Af_3(x)-Af_4(x))$, where each $f_i$ is non-negative and measurable.  
Hence, $Af$ is measurable since it is the sum of measurable functions on a measurable set.  Furthermore, since $|Af|\leq |f|$, (\ref{e.92371}) implies that
\begin{equation}
\int_X\left|Af(x)\right|^p \nu_1(dx)\leq \int_Y |f(y)|^p\nu_2(dy).\label{e.92372}
\end{equation}

We see, then, that we have a well-defined map $A$ taking a function $f$ on $%
Y $ satisfying (\ref{pfinite2}) and producing a function $Af$ on $X$ defined
$\nu _{1}$-almost everywhere and given by%
\[
(Af)(x)=\int_{Y}f(y)\mu^x(dy).
\]%
The space of functions satisfying (\ref{pfinite2}) is a vector space and the
map $A$ is linear. Furthermore, if $f(y)=0$ for $\nu _{2}$%
-almost every $y,$ then by (\ref{e.92372}) we see that $(Af)(x)$ will be
zero for $\nu _{1}$-almost every $x.$ Thus, if $f_{1}(y)=f_{2}(y)$ for $\nu
_{2}$-almost every $y,$ then $(Af_{1})(x)=(Af_{2})(x)$ for $\nu _{1}$-almost
every $x.$ Thus, we may interpret $A$ as a map from $L^{p}(Y,\nu _{2})$
to $L^{p}(X,\nu _{1}),$ where as usual the elements of $L^{p}$ are \textit{%
equivalence classes} of functions equal almost everywhere.  Eq. (\ref{e.92372}) implies the contractivity statement (\ref{contractivityp}) for $1\leq p<\infty$.

Finally, if $f\in L^\infty(Y,\nu_2)$, then $f\in L^p(Y,\nu_2)$ for all $p<\infty$, and the above argument implies that $||Af||_{L^p(X,\nu_1)}\leq ||f||_{L^p(Y,\nu_2)}$ for all $1\leq p<\infty$.  Taking $p\rightarrow\infty$ on both sides of this inequality yields the result for $p=\infty$ as well.
\end{proof}

For our purposes, the most useful application of the theorem above occurs in the following setting.  We are given a measurable space $(X,\Omega )$ and a semigroup of probability measures $p (s,t,x,\cdot)$ on $X$. We may think of these measures as the
transition probabilities of an $X$-valued Markov process, i.e. $p (s,t,x,E)$
represents the probability that the process will belong to $E$ at time $t,$
given that it is at the point $x$ at time $s.$ The semigroup property means
that for any set $E\in \Omega $ and $r\leq s\leq t$ we have%
\[
\int_{X}p (r,s,y,dx)p (s,t,x,E)=p (r,t,y,E).
\]%
We now choose a basepoint $o\in X$, a \textquotedblleft
basetime\textquotedblright\ $t_{0},$ and two other times $s$ and $t$ with $%
t_{0}<s<t.$ We can construct an example of the preceding framework as
follows. Set $Y=X$, $\nu _{1}(\cdot )=p (t_{0},s,o,\cdot )$,
and
\[\mu^x (\cdot )=p (s,t,x,\cdot ).\]
The Markov property implies that%
\begin{equation}
\int_X \mu^x (E)~\nu _{1}(dx)=\int_X p (s,t,x,E)p (t_{0},s,o,dx)=p
(t_{0},t,o,E),\label{e.markov}
\end{equation}
and so $\nu _{2}(\cdot )=p (t_{0},t,o,\cdot ).$

In order to apply Theorem \ref{t.558}, is it necessary to show that $x\rightarrow \mu^x(E)$ is $\Omega$-measurable for any $E\in \Xi$.  As Remark \ref{r.3.3.8} indicates, this is obvious when the collection of measures $\{\mu^x\}_{x\in X}$ has a density with respect to a fixed reference measure $\lambda$.  For example, if $X=Y$ is a normed vector space and $\lambda$ is a finite measure on $(X,\mathcal{B}(X))$, then
\[x\rightarrow\int_Y 1_E(x+y)\lambda(dy)\]
is measurable.  Hence the collection of measures $\{\mu^x\}_{x\in X}$, where
\[\mu^x(E)=\int_Y 1_E(x+y)\lambda(dy),\]
satisfies the hypotheses of Theorem \ref{t.558}.  In this case, the measure $\nu_2$ defined by Eq. (\ref{e.399}) is the convolution of $\lambda $ and $\nu_1$.

The first specific example is the context of Section 3.  The second example can be thought of an infinite dimensional example of the first and underscores the comments of Remark \ref{r.3.3.1}.
\begin{example}\label{e.887}
Let $\{X_t\}_{t\geq 0}$ denote a continuous-time Markov process on a smooth manifold $M$.  Fix $t>0$ and let $\{\mu_t^x\}_{x\in M}$ denote the collection of heat kernel measures on $M$.  Suppose that $x\rightarrow \mu_t^x(E)$ is measureable for all $E\in \mathcal{B}(M)$.  Remark \ref{r.3.3.8} indicates that this will be the case when the heat kernel measures are absolutely continuous with respect to a fixed reference measure, for example.  Now set $X=Y=M$ and consider $\nu_1=\mu_{T-t}^o$ for some $0\leq t<T$ and fixed $o\in M$.  Then the Markov property indicates
\[\nu_2(E)=\int_M\mu_t^x(E)\mu_{T-t}^o(dx)=\mu_T^o(E),\]
i.e. $\nu_2=\mu_T^o$.  By Theorem \ref{t.558}, the map
\[(Af)(x)=\int_M f(y)\mu_t^x(dy)\]
is a contraction from $L^p(M,\mu_T^o)$ to $L^p(M,\mu_{T-t}^o)$ for all $1\leq p\leq \infty$.
\end{example}

\begin{example}\label{e.3.3}
For any $t>0$, let $(W,H,\mu_t)$ denote an abstract Wiener space where $\mu_t$ is a Gaussian measure with variance $t$.  We can define the collection of measures $\{\mu_t\}_{t\geq 0}$ on $\mathcal{B}(W)$ by their characteristic functionals: for all $u\in W^*$,
\[\hat{\mu_t}(u):=\int_W e^{iu(w)}\mu_t(dw)=e^{-\frac{t}{2}(u,u)},\]
where $(\cdot,\cdot)$ is an inner product on $W^*$.  Note that for any $s,t>0$, $\mu_s\ast\mu_t=\mu_{s+t}$.

Set $X=Y=W$ and $\Omega=\Xi=\mathcal{B}(W)$.  Fix two times $0\leq t<T$ and set $\nu_1=\mu_{T-t}$ and
\[\mu^x(E)=\int_W1_{E}(x+y)\mu_{t}(dy)\]
for any $E\in \mathcal{B}(W)$ and $x\in W$.  Then $\mu^x$ is measurable in $x$ and  $\nu_2=\mu_{T-t}\ast\mu_t=\mu_T$.  It follows that for any $f\in L^p(W,\mu_T)$, the function
\[(Af)(x)=\int_W f(x+y)\mu_t(dy),\]
is defined for $\mu_{T-t}$-a.e. $x\in W$ and is in $L^p(W,\mu_{T-t})$.

If $f_1$ and $f_2$ are both representatives of the same $L^p(W,\mu_t)$ equivalence class, then the Cameron--Martin Theorem states that $(Af_1)(x)=(Af_2)(x)$ for all $x\in H$.  If, on the other hand, $f_1$ and $f_2$ are both representatives of the same $L^p(W,\mu_T)$ equivalence class, then Theorem \ref{t.558} indicates that $(Af_1)(x)=(Af_2)(x)$ for $\mu_{T-t}$-almost every $x\in W$.
\end{example}

The setup and hypotheses of Theorem \ref{t.558} are general enough to apply to many examples beyond the semigroups of probability measures described above.  For other examples of Markov kernels, see \cite{bauer}.

\section{Harmonic and Subharmonic Functions}

When $X=Y$ in Theorem \ref{t.558}, the operator $A$ is a map between functions defined on the same space.  In this case, it is possible for $A$ to fix some class of functions.  If $f$ is a function on $X$ such that $\int_X |f(y)|^p\nu_2(dy)<\infty$ for some $1\leq p\leq\infty$ and if $Af=f$ $\nu_1$-a.s., then Theorem \ref{t.558} immediately gives
\begin{equation}\label{e.490}
||f||_{L^p(X,\nu_1)}=||Af||_{L^p(X,\nu_1)}\leq ||f||_{L^p(X,\nu_2)}.
\end{equation}
When $A$ belongs to a semigroup of operators, $A=e^{tL}$ for some $t>0$ and some operator $L$, it is reasonable to expect that a function $f$ in the domain of $L$ satisfies (\ref{e.490}) if it is $L$-harmonic ($Lf=0$).  The subject of the following subsection is to prove this statement under some additional assumptions when the operator $L$ is the generator of a diffusion.  We then show that these conditions are satisfied in two examples: $L=\frac{1}{2}\Delta$, the Laplace--Beltrami operator on a Riemannian manifold with Ricci curvature bounded from below, and $L$ a subelliptic invariant operator on a Lie group satisfying H\"{o}rmander's condition.

\subsection{Diffusion Processes}

Throughout this section, we will assume that $M$ is a smooth manifold and $X=\{X_t\}_{t\geq 0}$ an $M$-valued diffusion with infinitesimal generator $L$ defined on the filtered probability space $(\Omega,P,\mathcal{F}_t)$.   This means that $\{X_t\}_{t\geq 0}$ is a time-homogeneous strong Markov process with continuous sample paths such that for any $f\in C^2(M)$,
\begin{equation}\label{e.143}
M_t^f:=f(X_t)-\int_0^t(Lf)(X_s)ds
\end{equation}
is a local martingale for $t<e(X)$, where $e$ denotes the explosion time.  We will assume that $e(X)=\infty$ a.s., i.e. $M$ is stochastically complete.  Let $\{\mu_t^x\}_{x\in M}$ denote the transition probabilities of this process, $\mu_t^x(E)=P_x(X_t\in E)$, where $P_x$ denotes the probability with respect to the initial distribution $X_0=x$ almost surely (this process will be denoted $\{X_t^x\}_{t\geq 0}$ when necessary).  Similarly, we will use $\mathbb{E}_x[f(X_t)]$ to denote expectation of $f(X_t)$, where $X_0=x$ almost surely; more generally, $x$ can be replaced by any initial distribution.  As before, we define the heat operator
\[(e^{tL}f)(x):=\int_M f(y)\mu_t^x(dy)=\mathbb{E}_x[f(X_t)].\]
As shown in Example \ref{e.887}, $e^{tL}$ is a contraction from $L^p(M,\mu_T^o)$ to $L^p(M,\mu_{T-t}^o)$ for any $o\in M$, $1\leq p\leq \infty$ and $0\leq t<T$.

A function $f\in C^2(M)$ is $L$-harmonic if $Lf=0$ and is $L$-subharmonic if $Lf\geq 0$.  If $L$ is hypoelliptic (as will be the case in the two examples that follow), then a function that is $L$-harmonic in the distributional sense will automatically be smooth; however, the same cannot be said of an $L$-subharmonic function.  Eq. (\ref{e.143}) implies that $\{f(X^x_t)\}_{t\geq 0}$ is a local martingale for all $x\in M$ if $f$ is $C^2$ and $L$-harmonic, while $\{f(X^x_t)\}_{t\geq 0}$ is a local submartingale if $f$ is $C^2$ and $L$-subharmonic.

In the sequel we will be careful to make a distinction between a function $f\in C^2(M)$ which is $L$-harmonic or $L$-subharmonic $f$ and satisfies the integrability condition
\[\int_M |f(y)|^p\mu_T^o(dy)<\infty\]
and its $L^p(M,\mu_T^o)$ equivalence class.  We will say $f\in L^p(M,\mu^o_T)$ is $C^2$ and $L$-harmonic (or $L$-subharmonic) if there exists a $C^2$ $L$-harmonic ($L$-subharmonic) representative in $L^p(M,\mu_T^o)$.

Let $d:M\times M\rightarrow \mathbb{R}_+$ denote a metric on $M$ such that closed balls
\[B(x,r)=\{y\in M| d(x,y)\leq r\}\]
are compact.  Let $\tau_{B(x,r)}$ denote the first exit time of $X$ from $B(x,r)$.  Note that since $M$ is stochastically complete, $P_x(\tau_{B(x,r)}\leq t)\rightarrow 0$ as $r\rightarrow\infty$ for any $x\in M$ and $t\geq 0$.

\begin{definition}
A process $\{X_t\}_{t\geq 0}$ is uniformly local with parameter $\delta>0$ if there exists $t,\epsilon >0$ such that for all $x\in M$,
\begin{equation}\label{e.994}
P_x(d(x,X_s)\leq\delta)\geq \epsilon
\end{equation}
whenever $0\leq s<t$.
\end{definition}
\begin{example}\label{e.4.33}
Suppose $\{X_t\}_{t\geq 0}$ is a left(right)-invariant strong Markov process on a Lie group $M$ endowed with left(right)-invariant metric $d$.  Then if $P_o(d(o,X_s)\leq\delta)\geq \epsilon$ for one particular $o\in M$, then $P_x(d(x,X_s)\leq\delta)\geq \epsilon$ for all $x\in M$.  In other words, $\{X_t\}_{t\geq 0}$ is uniformly local with parameter $\delta$ if (\ref{e.994}) is satisfied at any one point.
\end{example}
\begin{example}\label{e.4.34}
Suppose $\{X_t\}_{t\geq 0}$ is a Brownian motion on a Riemannian manifold of dimension $n$ with $\operatorname{Ric}(M)\geq -K$ for some $K\geq 0$.  Let $d$ denote the Riemannian distance.  Roughly speaking, a lower bound on the Ricci curvature of $M$ gives an upper bound on the rate of escape of Brownian motion.  One precise formulation of this idea is Theorem 3.6.1 of Hsu \cite{hsu_book}:
\begin{theorem}
Suppose $L\geq 1$ and $M$ is a Riemannian manifold of dimension $n$ with $Ric(z)\geq -L^2$ for all $z\in B(x,1)$.  Then there exists a constant $C$ depending only on $n$ such that
\[P_x(\tau_{B(x,1)}\leq \frac{C}{L})\leq e^{-L/2}.\]
\end{theorem}
We can assume, without loss of generality, that $K\geq 1$.  In which case, if we set $t=\frac{C}{\sqrt{K}}$, then for $s<t$ we have
\[P_x(d(X_s,x)\leq 1)\geq P_x(\tau_{B(x,1)}\geq s)\geq P_x(\tau_{B(x,1)}>t)>1-e^{-\sqrt{K}/2}.\]
It follows that $\{X_t\}_{t\geq 0}$ is uniformly local with $\delta=1$.
\end{example}
Clearly, if $\delta_1<\delta_2$, then a uniformly local processes with parameter $\delta_1$ is also a uniformly local process with parameter $\delta_2$.  Furthermore, if the process possesses a scaling property then the choice of $\delta$ is particularly arbitrary.  For example, if $\{B_t\}_{t\geq 0}$ is standard Brownian motion on $\mathbb{R}$ and $P_x(|B_s|<\delta)\geq \epsilon$ when $s<t$, then Brownian scaling indicates that whenever $s<\alpha^2t$ for some $0<\alpha<1$,
\[P_x(|B_s|\leq \alpha\delta)=P_x(|\frac{1}{\alpha}B_s|\leq \delta)=P_x(|B_{s/\alpha^2}|\leq \delta)\geq \epsilon.\]
 This implies that $\{B_t\}_{t\geq 0}$ is uniformly local with \textit{any} parameter $\delta>0$.  As this discussion indicates, one often has a choice as to what parameter $\delta$ to assign to a uniformly local process, however, the value of $\delta$ matters in the application of the following theorem, which is the primary result of this section.
\begin{theorem}\label{t.337}
Suppose $\{X_t\}_{t\geq 0}$ is uniformly local with parameter $\delta>0$ and $f\in C^2(M)$.  Let $t_0, T>0$.  Suppose that for every $x\in M$, there is a $\eta(x)>0$ such that the following conditions are satisfied:
\begin{enumerate}[(A)]
\item $\displaystyle \int_M|f(y)|^{1+\eta(x)}\mu_t^x(dy)<\infty$ for all $0\leq t<T$.
\item For any $t<t_0$,
\[\lim_{n\rightarrow\infty}\left(\sup_{y\in B(x,n)}|f(y)|\right)P_x\left(n-\delta\leq d(X_t,x)\leq n+\delta\right)^{\frac{1}{1+\eta(x)}}=0.\]
\end{enumerate}
Then the following statements hold:
\begin{enumerate}[(1)]
\item If $Lf=0$, then for all $0\leq t<T$
\[e^{tL}f=f.\]
\item If $Lf\geq 0$, then for all $0\leq t<T$
\[e^{tL}f\geq f.\]
\item If $Lf=0$, or if $Lf\geq 0$ and $f\geq 0$, and for some $1\leq p<\infty$ and $o\in M$
\[\int_M |f(y)|^p\mu_T^o(dy)<\infty,\]
then for any $0\leq s\leq  t<T$
\begin{equation*}
\int_M|f(y)|^p\mu_s^o(dy)\leq\int_M|f(y)|^p\mu_t^o(dy)\leq \int_M |f(y)|^p\mu_T^o(dy).
\end{equation*}
\end{enumerate}
\end{theorem}

In the applications of Theorem \ref{t.337} that follow, condition $(A)$ will follow from a suitable Harnack inequality, while condition $(B)$ will follow from estimates of the supremum of $f$ over balls in terms of the $L^p(M,\mu_t^x)$ norm of $f$ and suitable upper and lower bounds on the heat kernel in terms of the distance function. The proof of the above theorem appears at the end of this section and uses the following results.
\begin{proposition}\label{p.298}
Suppose $f\in C^2(M)$ such that
\[\int_M |f(y)|^{p}\mu_t^x(dy)<\infty\]
for some fixed $x\in M$, $1<p<\infty$ and $t\geq 0$, and suppose there exists a $C<\infty$ such that for all positive integers $n$,
\[\left(\sup_{y\in B(x,n)}|f(y)|\right)P_x(\tau_{B(x,n)}\leq t)^{\frac{1}{p}}<C.\]
Then $\mathbb{E}_x[f(X_t)]=(e^{tL}f)(x)=f(x)$ if $f$ if $Lf=0$ while $\mathbb{E}_x[f(X_t)]=(e^{tL}f)(x)\geq f(x)$ if $Lf\geq 0$.
\end{proposition}
\begin{proof}
Since $f\in C^2(M)$, it is continuous and hence bounded on $B(x,n)$.  Eq. (\ref{e.143}) implies that $\{f(X^x_t)\}_{t\geq 0}$ is a local (sub)martingale when $Lf=(\geq)0$, i.e. there exists an increasing sequence of stopping times $\{\sigma_m\}_{m\geq 1}$ such that $\{f(X^x_{t\wedge \sigma_m})\}_{t\geq 0}$ is a (sub)martingale for all $m$.  In addition, $\{f(X^x_{t\wedge \sigma_m\wedge\tau_{B(x,n)}})\}_{t\geq 0}$ is a (sub)martingale for all $n$ and $f(X^x_{t\wedge \sigma_m\wedge\tau_{B(x,n)}})$ converges a.s. and boundedly to $f(X^x_{t\wedge\tau_{B(x,n)}})$ for any fixed $t$ as $m\rightarrow \infty$.  It follows that
\begin{align*}
\mathbb{E}_x[f(X_{t\wedge \tau_{B(x,n)}})|\mathcal{F}_s]&=\lim_{m\rightarrow\infty}\mathbb{E}_x[f(X_{t\wedge \sigma_m\wedge\tau_{B(x,n)}})|\mathcal{F}_s]\\&=(\geq)\lim_{m\rightarrow\infty}f(X_{s\wedge \sigma_m\wedge\tau_{B(x,n)}})\\&=f(X_{s\wedge\tau_{B(x,n)}})
\end{align*}
where the first equality follows by the Dominated Convergence Theorem for conditional expectation and the second (in)equality is the (sub)martingale property of $\{f(X_{t\wedge \sigma_m\wedge\tau_{B(x,n)}})\}_{t\geq 0}$.  Therefore, the process $\{f(X^x_{t\wedge\tau_{B(x,n)}})\}_{t\geq 0}$ is a bounded (sub)martingale when $f$ is $L$-(sub)harmonic.  In particular,
\[\mathbb{E}_x[f(X_{t\wedge\tau_{B(x,n)}})]=(\geq)\mathbb{E}_x[f(X_0)]=f(x)\]
for all $t\geq 0$ and all $n$.  The remainder of the argument amounts to showing that $\lim_{n\rightarrow\infty}\mathbb{E}_x[f(X_{t\wedge\tau_{B(x,n)}})]=\mathbb{E}_x[f(X_t)]$ and is independent of whether the function $f$ is $L$-harmonic or $L$-subharmonic.

Note that if $t=0$, then the conclusion of the theorem follows immediately, and so we will assume $t$ is a fixed positive number.  Consider the sequence of random variables, now indexed by the positive integers, $\{f(X^x_{t\wedge\tau_{B(x,n)}})\}_{n\geq 0}$.  Clearly, $\mathbb{E}_x[|f(X_{t\wedge\tau_{B(x,n)}})|]<\infty$ for all $n$ since $f$ is bounded on $B(x,n)$.  Also, since $P_x(\tau_{B(x,n)}\leq t)\rightarrow 0$ as $n\rightarrow\infty$, we have $f(X^x_{t\wedge\tau_{B(x,n)}})\rightarrow f(X^x_t)$ in probability.  Note that
\begin{align*}
\mathbb{E}_x[|f(X_{t\wedge\tau_{B(x,n)}})|^{p}]^{\frac{1}{p}}&\leq \mathbb{E}_x[|f(X_{t})|^{p}1_{\{t< \tau_{B(x,n)}\}}]^{\frac{1}{p}}+\mathbb{E}_x[|f(X_{\tau_{B(x,n)}})|^{p}1_{\{\tau_{B(x,n)}\leq t\}}]^{\frac{1}{p}}\\ &\leq \mathbb{E}_x[|f(X_{t})|^{p}]^{\frac{1}{p}}+\left(\sup_{y\in B(x,n)}|f(y)|\right)P_x(\tau_{B(x,n)}\leq t)^{\frac{1}{p}}.
\end{align*}
Our assumption therefore guarantees that $\{f(X_{t\wedge\tau_{B(x,n)}}^x)\}_{n\geq 0}$ is bounded in $L^{p}(P)$.  This, along with the fact that $f(X_{t\wedge\tau_{B(x,n)}}^x)\rightarrow f(X^x_t)$ in probability, implies that $f(X_{t\wedge\tau_{B(x,n)}}^x)\rightarrow f(X^x_t)$ in $L^1(P)$ (see, for example, pg. 353 of \cite{grimmett}).  Hence,
\[\mathbb{E}_x[f(X_t)]=\lim_{n\rightarrow\infty}\mathbb{E}_x[f(X_{t\wedge\tau_{B(x,n)}})]=f(x).\]
\end{proof}

\begin{proposition}\label{p.299}
Suppose $f$ is a measurable function on $M$ and suppose that there exists $0<t< T$ such that for all $x\in M$
\begin{enumerate}
\item $\int_M |f(y)|\mu_r^x(dy)<\infty$ for all $0\leq r<T$
\item $(e^{sL}f)(x)=f(x)$ for all $0\leq s<t$.
\end{enumerate}
Then $(e^{sL}f)(x)=f(x)$ for all $x\in M$ and $0\leq s< T$. Similarly, if $(2)$ is replaced by $(e^{sL}f)(x)\geq f(x)$, then $(e^{sL}f)(x)\geq f(x)$ for all $x\in M$ and $0\leq s< T$.
\end{proposition}
\begin{proof}
If $t\leq s<2t\wedge T$, then $s/2<t$ and $(1)$ implies the use of Fubini's Theorem is justified in second line of the computation
\begin{align*}
(e^{sL}f)(x) &=\int_M f(y)\left(\int_M\mu_{s/2}^z(dy)\mu_{s/2}^x(dz)\right)\\ &=\int_M \mu_{s/2}^x(dz)\left(\int_M f(y)\mu_{s/2}^z(dy)\right)\\ &=\int_M f(z)\mu_{s/2}^x(dz)\\ &=f(x).
\end{align*}
The first equality follows from the Markov property, while assumption $(2)$ implies the last two equalities.  This establishes that $(e^{sL}f)(x)=f(x)$ for all $0\leq s<2t\wedge T$, and repeating this process as needed yields the first result for $0\leq s<T$.  The second result is obtained by changing the last two equalities into inequalities.
\end{proof}

\begin{lemma}\label{l.30}
Let $D,E\subset M$ denote a closed sets such that
\[P_y(X_t\in E)=\mathbb{E}_y[1_E(X_t)]\geq\epsilon\]
for all $y\in D$ and $0\leq t<T$.  If $\tau$ denotes the first hitting time of $D$ and $\tau<\infty$ a.s., then for all $x\in M$ and all $0\leq t<T$,
\[\mathbb{E}_x[1_E(X_t)|\mathcal{F}_\tau]\geq\epsilon 1_{\{\tau\leq t\}}\quad\text{P-a.s.}\]
\end{lemma}
\begin{proof}
We first prove the above result with $1_E$ replaced by a non-negative bounded continuous function.  To that end, suppose $g:M\rightarrow \mathbb{R}$ is a non-negative bounded continuous function such that $\mathbb{E}_y[g(X_t)]\geq\epsilon$
for all $y\in D$ and $0\leq t<T$.  It follows that $\mathbb{E}_{\nu}[g(X_t)]\geq\epsilon$ for any initial distribution $\nu$ supported in $D$; in particular $\mathbb{E}_{X_\tau^x}[g(X_t)]\geq\epsilon$ for any $x\in M$.

Set $\sigma=t-\tau$.  Observe that $0\leq \sigma< t$ on $\{\tau\leq t\}$ and
\[\mathbb{E}_x[g(X_t)|\mathcal{F}_\tau]\geq \mathbb{E}_x[1_{\{\tau\leq t\}}g(X_t)|\mathcal{F}_\tau]=\mathbb{E}_x[1_{\{\tau\leq t\}}g(X_{\tau+\sigma})|\mathcal{F}_\tau].\]
For a positive integer $n$, let $D_n$ denote the dyadic approximation on $\mathbb{R}$ given by $D_n(t)=\frac{k}{2^n}$ when $\frac{k}{2^n}\leq t<\frac{k+1}{2^n}$.  $D_n$ is right-continuous and measurable, and $\lim_{n\rightarrow\infty}D_n(t)=t$ for all $t$.  Consider the random variables $\sigma_n:=D_n\circ\sigma$.  Note that $\sigma_n\rightarrow \sigma$ almost surely and for any $k$,
\[\{\sigma_n=\frac{k}{2^n}\}=\{\frac{k}{2^n}\leq \sigma< \frac{k+1}{2^n}\}=\{t-\frac{k+1}{2^n}<\tau\leq t-\frac{k}{2^n}\}\in\mathcal{F}_\tau.\]
On $\{\tau\leq t\}$, $\sigma_n$ takes on only the finitely many positive values $\{0,\frac{1}{2^n},\frac{2}{2^n},...,\frac{N}{2^n}\}$.  It follows that
\[1_{\{\tau\leq t\}}g(X_{\tau+\sigma_n})=1_{\{\tau\leq t\}}\left(\sum_{k=0}^N 1_{\{\sigma_n=\frac{k}{2^n}\}}g(X_{\tau+\frac{k}{2^n}})\right)\]
for all $n$.  Now using the $\mathcal{F}_\tau$-measurability of $\{\tau\leq t\}$ and $\{\sigma_n=\frac{k}{2^n}\}$ and the strong Markov property, we see
\begin{align*}
\mathbb{E}_x[1_{\{\tau\leq t\}}g(X_{\tau+\sigma_n})|\mathcal{F}_\tau] &=1_{\{\tau\leq t\}}\left(\sum_{k=0}^N 1_{\{\sigma_n=\frac{k}{2^n}\}}\mathbb{E}_x[g(X_{\tau+\frac{k}{2^n}})|\mathcal{F}_\tau]\right)\\ &=1_{\{\tau\leq t\}}\left(\sum_{k=0}^N 1_{\{\sigma_n=\frac{k}{2^n}\}}\mathbb{E}_{X_\tau^x}[g(X_{\frac{k}{2^n}})|\mathcal{F}_\tau]\right)\\&\geq \epsilon 1_{\{\tau\leq t\}}.
\end{align*}
Since $g$ is continuous, $g(X_{\tau+\sigma_n}^x)\rightarrow g(X_t)$ a.s.  Since $g$ is bounded, the Dominated Convergence Theorem for conditional expectation gives
\[\mathbb{E}_x[g(X_t)|\mathcal{F}_\tau]\geq \mathbb{E}_x[1_{\{\tau\leq t\}}g(X_t)|\mathcal{F}_\tau]=\lim_{n\rightarrow\infty}\mathbb{E}_x[1_{\{\tau\leq t\}}g(X_{\tau+\sigma_n})|\mathcal{F}_\tau]\geq \epsilon 1_{\{\tau\leq t\}}.\]
We now construct a sequence of bounded continuous functions $\{g_n\}$ which decrease everywhere to $1_E$.  Since for each $n$, $g_n\geq 1_E$, it follows that $\mathbb{E}_x[g_n(X_s)]\geq\mathbb{E}_x[1_E(X_s)]\geq\epsilon$.  The above argument implies that $\mathbb{E}_x[g_n(X_t)|\mathcal{F}_\tau]\geq\epsilon 1_{\{\tau\leq t\}}$ $P$-a.s. for each $n$.  The desired result now follows by again applying the Dominated Convergence Theorem for conditional expectation.
\end{proof}

\begin{proof}[Proof of Theorem \ref{t.337}]
Fix $x\in M$ and $n\geq 0$.  For notational simplicity, we will set $\tau=\tau_{B(x,n)}$.  Set $E=\{y\in M|n-\delta\leq d(x,y)\leq n+\delta\}$ so that
\[P_x\left(n-\delta\leq d(x,X_t)\leq n+\delta\right)= \mathbb{E}_x[1_E(X_t)].\]
Since $\{X_t\}_{t\geq 0}$ is uniformly $\delta$-local, there exists a $t_1$ and an $0<\epsilon\leq 1$ such that $P_{y}[X_t\in E]=\mathbb{E}_y[1_E(X_t)]\geq \epsilon$ for all $y\in \partial B(x,n)$ whenever $t<t_1$.  Lemma \ref{l.30} then implies that $\mathbb{E}_x[1_E(X_t)|\mathcal{F}_\tau]\geq \epsilon 1_{\{\tau\leq t\}}$ a.s. whenever $t<t_1$.  It follows that
\[\mathbb{E}_x[1_E(X_t)]=\mathbb{E}_x[\mathbb{E}[1_E(X_t)|\mathcal{F}_\tau]]\geq \epsilon P_x(\tau\leq t),\]
giving
\[P_x(\tau\leq t)\leq \epsilon^{-1}P_x\left(n-\delta\leq d(x,X_t)\leq n+\delta\right),\]
and in particular
\begin{align*}
P_x(\tau\leq t)^{\frac{1}{1+\eta(x)}}&\leq \epsilon^{-\frac{1}{1+\eta(x)}}P_x\left(n-\delta\leq d(x,X_t)\leq n+\delta\right)^{\frac{1}{1+\eta(x)}}\\ &\leq \epsilon^{-1}P_x\left(n-\delta\leq d(x,X_t)\leq n+\delta\right)^{\frac{1}{1+\eta(x)}} ,
\end{align*}
for $t<t_1$.

Now for any $t<t_1\wedge t_0\wedge T$, we have
\begin{align*}&\left(\sup_{y\in B(x,n)}|f(y)|\right)P_x(\tau\leq t)^{\frac{1}{1+\eta(x)}}\\&\leq \frac{1}{\epsilon}\left(\sup_{y\in B(x,n)}|f(y)|\right)P_x\left(n-\delta\leq d(x,X_t)\leq n+\delta\right)^{\frac{1}{1+\eta(x)}}
\end{align*}
which tends to $0$ as $n\rightarrow\infty$ by assumption $(B)$ of Theorem \ref{t.337}.  Therefore, applying Proposition \ref{p.298} with $p=1+\eta(x)$ implies that $(e^{tL}f)(x)=f(x)$ for all $t<t_1\wedge t_0\wedge T$ if $Lf=0$, while $(e^{tL}f)(x)\geq f(x)$ for all $t<t_1\wedge t_0\wedge T$ if $Lf\geq0$.  This $t$ is independent of $x$, so Proposition \ref{p.299} gives that $e^{tL}f=f$ for all $0\leq t<T$ when $Lf=0$, while $e^{tL}f\geq f$ for all $0\leq t<T$ when $Lf\geq0$.  This proves statements $(1)$ and $(2)$ of Theorem \ref{t.337}.  Statement $(3)$ follows from these statements and Theorem \ref{t.558}.  Specifically, if $0\leq s\leq t<T$ and $Lf=0$, then
\begin{equation}\label{e.3030}
\int_M|f(y)|^p\mu_s^x(dy)=\int_M|(e^{(t-s)L}f)(y)|^p\mu_s^x(dy)\leq \int_M |f(y)|^p\mu_t^x(dy)
\end{equation}
and
\begin{equation}\label{e.3040}
\int_M|f(y)|^p\mu_t^x(dy)=\int_M|(e^{(T-t)L}f)(y)|^p\mu_t^x(dy)\leq \int_M |f(y)|^p\mu_T^x(dy),
\end{equation}
where the first equalities in (\ref{e.3030}) and (\ref{e.3040}) follow from $(1)$ of Theorem \ref{t.337} and the second equalities from Theorem \ref{t.558}.  If $Lf\geq 0$ and $f\geq 0$, then $|f(y)|\leq |(e^{sL}f)(y)|$ by $(2)$ and first equalities (\ref{e.3030}) and (\ref{e.3040}) are replaced by inequalities.
\end{proof}

\subsection{The Laplace--Beltrami on Riemannian Manifolds with Ricci Curvature Bounded from Below}\label{s.3.2}

Let $M$ denote a Riemannian manifold of dimension $D>1$ with Ricci curvature bounded from below, i.e. there exits a $K\geq 0$ such that $Ric \geq -K$.  Some necessary estimates will be phrased in terms of the non-negative quantity $k=\frac{K}{(D-1)}$.  Let $d$ denote the Riemannian distance and $V$ the Riemannian volume measure.  Let $\{X_t\}_{t\geq 0}$ denote a Brownian motion on $M$ with generator $\frac{1}{2}\Delta$, where $\Delta$ denotes the Laplace--Beltrami operator.  The transition probabilities $\mu_t^x$ are absolutely continuous with respect to the Riemannian volume measure; the density is the heat kernel $\mu_t(x,y):=\frac{\mu_t^x(dy)}{V(dy)}$.  It is well known that such a manifold is stochastically complete.  Example \ref{e.4.34} indicates that $\{X_t\}_{t\geq 0}$ is uniformly local with parameter $\delta=1$.

Let $f\in C^2(M)$ such that $\Delta f\geq 0$ and
\[\mathbb{E}_o[|f(X_T)|^p]=\int_M |f(y)|^p\mu_T(o,y)V(dy)<\infty\]
for some fixed $o\in M$ and $1<p<\infty$.  In this subsection, we will show that the conditions $(A)$ and $(B)$ of Theorem \ref{t.337} apply to such an $f$, which, along with an addition argument when $t=T$, amounts to a proof of Theorem \ref{t.888}.

\begin{remark}
Many of our references for estimates on the heat kernel consider the heat equation $\frac{\partial u}{\partial t}=\Delta u$, which differs by a factor of $\frac{1}{2}$ from our heat equation.  As a result, the estimates below differ from the statements in the references by this factor.
\end{remark}

We begin by recalling the Li--Yau Harnack inequality as stated in Theorem 5.3.5 of Davies \cite{davies_heatkernels} with $\alpha=2$ applied to the function $u(t,\cdot)=\mu_t(\cdot,y)$:  for any $0<t<T$ and $x,z\in M$,
\[0\leq \mu_t(x,y)\leq \mu_T(z,y)\left(\frac{T}{t}\right)^D\exp{\left(\frac{d(x,z)^2}{T-t}+\frac{DK(T-t)}{4}\right)}.\]
This implies that the function $y\rightarrow \frac{\mu_t(x,y)}{\mu_T(z,y)}$ is bounded for any $x,z\in M$.  It follows that for any $x\in M$ and $0<t<T$,
\begin{align*}
\int_M |f(y)|^p\mu_t(x,y)V(dy) &=\int_M |f(y)|^p\left(\frac{\mu_t(x,y)}{\mu_T(o,y)}\right)\mu_T(o,y)V(dy)\\&<C(x)\int_M |f(y)|^p\mu_T(o,y)V(dy)\\& <\infty.
\end{align*}
Therefore condition $(A)$ of Theorem \ref{t.337} is satisfied for all $0< t<T$ with $\eta(x)=p-1>0$.  When $t=0$, condition $(A)$ is satisfied as well since $f$ is everywhere finite.

To see that condition $(B)$ of Theorem \ref{t.337} is satisfied, we begin with an inequality from Theorem 2.1 of Li \& Schoen \cite{li_schoen}.  Let $V(x,r)$ denotes the volume of the geodesic ball of radius $r$ centered at $x\in M$.  There exist constants $\alpha$ and $C$ which depend on $k$ but not $x$ such that for any non-negative subharmonic function $f$
\[\sup_{y\in B(x,r)}|f(y)|\leq Ce^{\alpha r}V(x,2r)^{-1}\int_{B(x,2r)}|f(y)|V(dy).\]
We modify this inequality to the following: for any positive integer $n$
\begin{align*}
\sup_{y\in B(x,n)}|f(y)|&\leq Ce^{\alpha n}V(x,2n)^{-1}\int_{B(x,2n)}|f(y)|\frac{\mu_T(x,y)}{\mu_T(x,y)}V(dy)\\ &\leq \frac{Ce^{\alpha n}}{V(x,2n)\left(\inf_{y\in B(x,2n)}\mu_T(x,y)\right)}\int_{B(x,2n)}|f(y)|\mu_T(x,y)V(dy)\\ &\leq \frac{Ce^{\alpha n}}{V(x,2)\left(\inf_{y\in B(x,2n)}\mu_T(x,y)\right)}||f||_{L^1(M,\mu^x_T)}.
\end{align*}
The lower bounds of Davies and Mandouvalos \cite{davies_mandouvalos} on hyperbolic space along with the Cheeger and Yau's comparison (see Wang \cite{wang_lower}) allows one to obtain a lower bound on the heat kernel
\begin{equation}\label{e.800}
\inf_{y\in B(x,2n)}\mu_T(x,y)\geq C(D,T)\exp{\left(-\frac{2n^2}{T}-\frac{(D-1)^2kT}{8}-(D-1)\sqrt{k}n\right)}.
\end{equation}
Putting the above estimates together, we see that there are constants $c_1,c_2,c_3,c_4$ (independent of $x$) such that for any non-negative subharmonic function $f$,
\begin{equation}\label{e.444}
\sup_{y\in B(x,n)}|f(y)|\leq c_1V(x,2)^{-1}\exp{\left(\frac{c_2n^2}{T}+c_3n+c_4T\right)}||f||_{L^1(\mu^x_T)}.
\end{equation}
Now for $t<T$ and $n>1$,
\begin{align*}
P_x(n-1\leq d(X_t,x)\leq n+1) &\leq \left(\sup_{y\in B(x,n+1)}\mu_t(x,y)\right)V(x,n+1).
\end{align*}
We use upper bounds on the heat kernel due to Li and Yau \cite{li_yau} as stated in Davies \cite{davies_state} with $\delta=\frac{1}{2}$:
\[\mu_t(x,y)\leq c_5V(x,(t/2)^{1/2})^{-1/2}V(y,(t/2)^{1/2})^{-1/2}\exp{\left(\frac{t}{4}-\frac{4d(x,y)^2}{9t}\right)}.\]
The relative volume comparison gives constants $c_6,c_7$ such that for all $r>0$,
\[V(x,r)\leq V(y,d(x,y)+r)\leq c_6V(y,r)r^{-D}\exp{\left(c_7\left(d(x,y)+r\right)\right)}.\]
This allows us to rewrite the heat kernel upper bounds as
\begin{align}
\mu_t(x,y)&\leq c_5c_6^{1/2}V(x,(t/2)^{1/2})^{-1}\left(\frac{t}{2}\right)^{-D/4}\nonumber\\& \times \exp{\left(\frac{t}{4}+\frac{c_7t^{1/2}}{\sqrt{8}}-\frac{4d(x,y)^2}{9t}+\frac{c_7d(x,y)}{2}\right)}\label{e.303}
\end{align}
By Bishop's comparison theorem, there are constants $c_8$ and $c_{9}$ depending on $k$ such that $V(x,r)\leq c_8e^{c_{9}r}$, and so if we set
\[C(x,t):=c_5c_6^{1/2}c_8V(x,(t/2)^{1/2})^{-1}\left(\frac{t}{2}\right)^{-D/4}\exp{\left(\frac{t}{4}+\frac{c_7t^{1/2}}{\sqrt{8}}\right)}\]
then
\begin{align*}
&P_x(n-1\leq d(X_t,x)\leq n+1) \\& \leq C(x,t) \exp{\left(-\frac{4}{9t}(n+1)^2+\frac{c_7+2c_9}{2}(n+1)\right)}\\ &=C(x,t)\exp{\left(\frac{c_7+2c_9}{2}-\frac{4}{9t}\right)}\exp{\left(-\frac{4}{9t}n^2+\left(\frac{c_7+2c_9}{2}-\frac{8}{9t}\right)n\right)}.
\end{align*}

By combining the estimate above with (\ref{e.444}), we see that there are constants $c_{10}$ and $c_{11}$, which depend on $x,p$ and $t$ but not $n$, such that \[\left(\sup_{y\in B(x,n)}|f(y)|\right)P_x(n-1\leq d(X_t,x)\leq  n+1)^{\frac{1}{p}}\leq c_{10}\exp{\left(\left(\frac{c_2}{T}-\frac{4}{9p t}\right)n^2+c_{11}n\right)}.\]
Set $t_0:=\frac{4T}{9p c_2}$.  If $t<t_0$, then $\frac{c_2}{T}-\frac{4}{9p t}<0$ and
\[\left(\sup_{y\in B(x,n)}|f(y)|\right)P_x(n-1\leq d(X_t,x)\leq  n+1)^{\frac{1}{p}}\rightarrow 0\qquad\text{ as }n\rightarrow\infty.\]
Therefore, condition $(B)$ of Theorem \ref{t.337} is satisfied.

Since we assume that $f\in L^p(M,\mu_T^o)$ for $1<p<\infty$, the following proposition indicates that $f\in L^r(M,\mu_{T+\epsilon}^o)$ for some small $\epsilon>0$ and $1<r<p$.  Therefore, an application of the results proved above show that, since $T<T+\epsilon$, $e^{T\Delta/2}f=f$ when $f$ is harmonic while $e^{T\Delta/2}f\geq f$ when $f$ is $C^2$ and subharmonic, and the proof of Theorem \ref{t.888} is complete.
\begin{proposition}
Suppose $f\in L^p(M,\mu_T^o)$ for some $1<p<\infty$.  Then for any $1<r<p$, there exists an $\epsilon>0$ such that $f\in L^r(M,\mu_{T+\epsilon}^o)$.
\end{proposition}
\begin{proof}
The proof is elementary and relies of the fact that we have upper and lower bounds on the heat kernel of essentially the same exponential order.  We observe first that the function $|f|^r\in L^\frac{p}{r}(M,\mu_T^o)$ since $|||f|^r||_{L^{\frac{p}{r}}(M,\mu_T^o)}=||f||_{L^p(M,\mu_T^o)}^r$.  It follows that
\begin{align}
\int_M |f(y)|^r\mu_{T+\epsilon}(o,y)V(dy)&=\int_M |f(y)|^r\frac{\mu_{T+\epsilon}(o,y)}{\mu_T(o,y)}\mu_T(o,y)V(dy)\nonumber\\&\leq |||f|^r||_{L^{\frac{p}{r}}(M,\mu_T^o)}||\frac{\mu_{T+\epsilon}(o,\cdot)}{\mu_T(o,\cdot)}||_{L^q(M,\mu_T^o)}\label{e.440}
\end{align}
where $q>1$ denotes the conjugate exponent to $\frac{p}{r}>1$.  Now pick $0<\delta<1$ and $\epsilon>0$ such that
\[\frac{4T}{(4+\delta)(T+\epsilon)}>\frac{q-1/2}{q}.\]
Note that one can always find such a $\delta$ and $\epsilon$ since the right-hand side of the above is less than one while the left-hand side increases to one as $\delta$ and $\epsilon$ both tend towards zero.  By the heat kernel bounds of Li and Yau \cite{li_yau} as stated in Davies \cite{davies_state},
\[\mu_{T+\epsilon}(o,y)\leq C\exp{\left(-\frac{2d(o,y)^2}{(4+\delta)(T+\epsilon)}\right)},\]
where the constant $C$ depends on $T+\epsilon, o, D, k$ and also exponentially on $d(o,y)$ (see line (\ref{e.303}) above).  The heat kernel lower bounds (\ref{e.800}) above imply that
\[\mu_T(o,y)\geq \tilde{C}\exp{\left(-\frac{d(o,y)^2}{2T}\right)},\]
where again $\tilde{C}$ depends on many factors, but only at most exponentially in $d(o,y)$.  For our chosen $\delta$ and $\epsilon$, these heat kernel bounds imply that the function $y\rightarrow \frac{\mu_{T+\epsilon}(o,y)^{q}}{\mu_T(o,y)^{q-1/2}}$ tends to zero as $d(o,y)\rightarrow \infty$; in particular, it is globally bounded by some constant $K$.  It follows that
\begin{align*}
||\frac{\mu_{T+\epsilon}(o,\cdot)}{\mu_T(o,\cdot)}||_{L^q(M,\mu_T^o)}^q&=\int_M \frac{\mu_{T+\epsilon}(o,y)^{q}}{\mu_T(o,y)^{q-1/2}}\mu_T(o,y)^{1/2}V(dy)\\&\leq K\int_M \mu_T(o,y)^{1/2}V(dy)\\&<\infty,
\end{align*}
and hence $f\in L^r(M,\mu_{T+\epsilon}^o)$ by line (\ref{e.440}).
\end{proof}
\subsection{Subelliptic Operators on Lie Groups}\label{s.3.3}

Suppose $M=G$ is a finite dimensional Lie group with a left-invariant Riemannian metric.  Let $\lambda$ denote a right-invariant Haar measure.  Suppose $\{b_t^k\}_{t\geq 0, 1\leq k\leq d}$ is a collection of independent $\mathbb{R}$-valued Brownian motions and let $\{X^x_t\}_{t\geq 0}$ denote the solution to the $G$-valued stochastic differential equation
\[dX_t^x=L_{X_t^x*}\left(\sum_{k=1}^d Y_k\delta b^k_t\right)\qquad X_0^x=x\text{ a.s.}\]
where $t\rightarrow \delta b_t^k$ denotes the Stratonovich differential, $\{Y_k\}_{k=1}^d\subset\mathfrak{g}=\operatorname{Lie}(G)$, and $L_{g*}$ denotes the differential of left translation $L_g(x)=gx$.   The process $\{X_t\}_{t\geq 0}$ has infinitesimal generator
\[L=\frac{1}{2}\sum_{k=1}^d \widetilde{Y_k}^2,\]
where $\widetilde{Y}$ denotes the left-invariant extension of $Y\in \mathfrak{g}$.  We also assume that $\{Y_k\}_{k=1}^d$ satisfies H\"{o}rmander's condition, i.e. $H$, the linear span of $\{Y_k\}_{k=1}^d$, generates $\mathfrak{g}$ under iterated Lie brackets.  Under this assumption, the left-invariant operator $L$ is hypoelliptic and the measure $\mu_t^x(E)=P_x(X_t\in E)$ has a smooth positive density with respect to $\lambda$, the heat kernel, which we denote by $\mu_t(x,y)=\frac{\mu_t^x(dy)}{\lambda(dy)}$.  The heat kernel is left-invariant, i.e. $\mu_t(x,y)=\mu_t(gx,gy)$ for all $x,y,g\in G$.

Let $d_H:G\times G\rightarrow\mathbb{R}$ denote the horizontal distance on $G$.  $d_H$ is defined by
\[d_H(x,y)=\inf_\gamma\left(\int_0^1 |L_{\gamma(t)^{-1}*}\gamma'(t)|dt\right),\]
where the infimum is taken over all horizontal paths, that is, all absolutely continuous paths $\gamma:[0,1]\rightarrow G$ such that $\gamma(0)=x$, $\gamma(1)=y$, and $L_{\gamma(t)^{-1}*}\gamma'(t)\in H$ for all $t\in (0,1)$.  The quantity
\[l_H(\gamma)=\int_0^1 |L_{\gamma(t)^{-1}*}\gamma'(t)|dt\]
is the horizontal length of $\gamma$.  We consider balls of radius $r>0$
\[B(x,r):=\{y\in G|d_H(x,y)\leq r\}\]
It is known that $B(x,r)$ is compact for all $x$ and $r\geq 0$.  As before, we will let $V(x,r)$ denote the volume of $B(x,r)$.  If $d(x,y)$ denotes the Riemannian distance between $x$ and $y$, then $d(x,y)\leq d_H(x,y)$.  If follows that $B(x,r)$ is contained within the Riemannian ball of radius $r$; in particular, there exists constants $b,k$ such that
\[V(x,r)\leq be^{kr}\]
for all $x\in G$.

If we fix $t>0$, and set
\[\epsilon=\inf_{0\leq s<t}\left(\int_{B(x,1)} \mu_s(x,y)\lambda(dy)\right)>0\]
for some $x\in G$, then by the comments of Example (\ref{e.4.33}), $\{X_t\}_{t\geq 0}$ is uniformly local with parameter $\delta=1$.

Suppose $L f=0$ and
\[\mathbb{E}_o[|f(X_T)|^p]=\int_G |f(y)|^p\mu_T(o,y)\lambda (dy)<\infty\]
for some fixed $o\in G$ and $1< p<\infty$.  Note that since $L$ is hypoellipitic, $f$ is necessarily smooth.  In this subsection, we will show that the conditions $(A)$ and $(B)$ of Theorem \ref{t.337} apply to such an $f$, which amounts to a proof of Theorem \ref{t.1.4}.

We first recall some necessary upper and lower bounds for the heat kernel.  For proofs and additional references, we refer the reader to Section 3 of Driver, Gross, \& Saloff-Coste \cite{dgsc_subelliptic}.  It should be noted, however, that \cite{dgsc_subelliptic} considers kernels corresponding to the operator $\frac{1}{2}L$, and so some statements below differ slightly from those in \cite{dgsc_subelliptic}.
\begin{proposition}\label{p.346}
For all $\kappa\in (0,1)$, there exists a positive constant $C_\kappa$ and an integer $\nu$ such that for all $x,y\in G$ and $t>0$
\[\mu_t(x,y)\leq C_\kappa\left(1+\frac{1}{2t}\right)^{\nu/2}e^{C_\kappa t}e^{-\kappa d_H(x,y)^2/2t}.\]
Also, there exists positive constants $C,c$ such that for all $x,y\in G$ and $t> 0$,
\[\mu_t(x,y)\geq c\left(1+\frac{1}{2t}\right)^{\nu/2}e^{-Ct}e^{-Cd_H(x,y)^2/2t}.\]
\end{proposition}
\begin{proposition}\label{p.347}
There exists a constant $K>0$ such that for $0<s<t$,
\[\mu_s(x,y)\leq \mu_t(z,y)\exp{\left(K\left(\frac{t}{s}+\frac{d_H(x,z)^2}{2(t-s)}\right)\right)}.\]
\end{proposition}
The parabolic Harnack inequality above (Proposition \ref{p.347}) implies that, for any $t<T$ and $x,o\in G$, the function $\frac{\mu_t(x,\cdot)}{\mu_T(o,\cdot)}$ is bounded.  Therefore, for any $0<t<T$ and $x\in M$,
\begin{align*}
\int_G |f(y)|^p\mu_t(x,y)\lambda(dy)&=\int_G |f(y)|^p\frac{\mu_t(x,y)}{\mu_T(o,y)}\mu_T(o,y)\lambda(dy)\\&\leq \exp{\left(K\left(\frac{T}{t}+\frac{d_H(x,o)^2}{2(T-t)}\right)\right)}\int_G |f(y)|^p\mu_T(o,y)\lambda(dy)\\&<\infty.
\end{align*}
Therefore, condition $(A)$ of Theorem \ref{t.337} is satisfied for all $0<t<T$ with $\eta(x)=p-1>0$.  When $t=0$, condition $(A)$ is satisfied as well since $f$ is everywhere finite.

To see that condition $(B)$ of Theorem \ref{t.337} is satisfied, we first obtain a pointwise bound on $L$-harmonic functions.
\begin{proposition}\label{p.592}
Suppose $Lf=0$ and $\int_G|f(y)|\mu_T(x,y)dy<\infty$ for some $x\in G$.  Then for any $\alpha>0$, there exists constants $D=D(\alpha,T) >0$ and $\beta>0$ such that
\[\sup_{y\in B(x,n)}|f(y)|\leq D e^{\frac{C}{T}(n+\alpha)^2+\beta(n+d_H(e,x))}||f||_{L^1(G,\mu_T^x)}.\]
\end{proposition}
\begin{proof}
Since $L$ is hypoelliptic and $Lf=0$, it follows (see Corollary III.1.3 of \cite{vscc} for example) that there exists a constant $D_0=D_0(\alpha)$ such that
\[|f(e)|\leq D_0||f||_{L^1(B(e,\alpha),\lambda)}.\]
Observe that the function $f\circ L_g$ is also $L$-harmonic for any $g\in G$ since $L$ is left-invariant.  So
\[|f(g)| \leq D_0||f\circ L_g||_{L^1(B(e,\alpha),\lambda)}=D_0m(g^{-1})||f||_{L^1(B(g,\alpha),\lambda)}\]
where $m$ denotes the modular function on $G$.  Now using the heat kernel lower bounds from Proposition \ref{p.346}
\begin{align*}
||f||_{L^1(B(g,\alpha),\lambda)} &=\int_{B(g,\alpha)}|f(y)|\frac{\mu_T(x,y)}{\mu_T(x,y)}\lambda(dy)\\ &\leq \left(\inf_{y\in B(g,\alpha) } \mu_T(x,y)\right)^{-1} ||f||_{L^1(G,\mu_T^x)}\\ &\leq ae^{\frac{C}{2T}(d_H(x,g)+\alpha)^2}||f||_{L^1(G,\mu_T^x)},
\end{align*}
for the constant $a^{-1}=c\left(1+\frac{1}{2T}\right)^{\nu/2}e^{-CT}$.  It follows that
\[\sup_{g\in B(x,n)}|f(g)|\leq aD_0 e^{\frac{C}{2T}(n+\alpha)^2}||f||_{L^1(G,\mu_T^x)}\left(\sup_{g\in B(x,n)} m(g^{-1})\right).\]
It remains to show that the modular function has at most exponential growth in the horizontal distance.

Let $\gamma$ denote a horizontal path such that $\gamma(0)=e$ and $\gamma(1)=g^{-1}$.  Define a sequence of times $t_0,t_1,t_2,...,t_r$ by setting $t_0=0$ and $t_{k+1}$ to be the first time $s>t_k$ such that $d_H(\gamma(t_k),\gamma(s))=1$.  Note that $\lfloor l_H(\gamma)\rfloor=r$ and that $d_H(\gamma(t_r),g^{-1})<1$.  Then \[g^{-1}=\gamma(t_0)\gamma(t_1)^{-1}\gamma(t_1)\gamma(t_2)^{-1}\gamma(t_2)\hdots \gamma(t_r)\gamma(t_r)^{-1}g^{-1},\]
and each $\gamma(t_k)\gamma(t_{k+1})^{-1}$ is contained in $B(e,1)$.  The modular function is a homomorphism, and so
\[m(g^{-1})=m(\gamma(t_0)\gamma(t_1)^{-1})m(\gamma(t_1)\gamma(t_2)^{-1})\hdots m(\gamma(t_r)^{-1}g^{-1}).\]
If we set $\theta=\sup_{y\in B(e,1)}m(y)$, then we see that
\[m(g^{-1})\leq \theta^{l_H(\gamma)+1}.\]
Taking the infimum over all such $\gamma$ gives that $m(g^{-1})\leq \theta^{d_H(e,g^{-1})+1}$.  It follows that $\sup_{g\in B(x,n)} m(g^{-1})\leq \theta^{d_H(e,x)+n+1}$.  The statement of the proposition follows with $D=aD_0\theta$ and $\beta=\ln{\theta}$.
\end{proof}

For any positive integer $n$,
\begin{align*}
P_x(n-1\leq d_H(X_t,x)\leq n+1) &\leq \left(\sup_{y\in B(x,n+1)}\mu_t(x,y)\right)V(x,n+1).
\end{align*}
The heat kernel upper bounds of Proposition \ref{p.346} above with $\kappa=\frac{1}{2}$ give
\[\sup_{y\in B(x,n+1)}\mu_t(x,y)\leq C_\kappa\left(1+\frac{1}{2t}\right)^{\nu/2}e^{C_\kappa t}e^{-(n+1)^2/4t}.\]
Along with the volume estimates on balls, we get that
\begin{align*}
P_x(n-1\leq d_H(X_t,x)\leq n+1)& \leq C_\kappa b\left(1+\frac{1}{2t}\right)^{\nu/2}e^{C_\kappa t}e^{k(n+1)} e^{-(n+1)^2/4t}.
\end{align*}
Putting this together with Proposition \ref{p.592} ($\alpha=1$) gives
\begin{align*}
\left(\sup_{g\in B(x,n)}|f(g)|\right)P_x(n-1\leq d_H(X_t,x)\leq n+1)^{\frac{1}{p}} &\leq \tilde{C} e^{\left(\frac{C}{T}-\frac{1}{4p t}\right)(n+1)^2+(\beta+\frac{k}{p}) n},
\end{align*}
where
\[\tilde{C}=\left(C_\kappa b\left(1+\frac{1}{2t}\right)^{\nu/2}e^{C_\kappa t+k}\right)^{\frac{1}{p}}De^{\beta d_H(e,x)}||f||_{L^1(G,\mu_T^x)}.\]
It follows now that if we chose $t<t_0=\frac{T}{4p C}$, then $\frac{C}{T}-\frac{1}{4p t}<0$, and this expression goes to zero as $n\rightarrow\infty$.  This proves condition $(B)$ of Theorem \ref{t.337}.

\appendix
\section{Infinite-dimensional applications}

We now briefly discuss some of the hoped-for infinite-dimensional
applications of our results. The first application is to give a meaning to
the heat operator on certain infinite-dimensional manifolds and groups. There are by
now many infinite-dimensional manifolds for which a based heat kernel \textit{%
measure} $\mu^o _{T}$ is known to exist, usually constructed as the
distribution of an appropriate stochastic process.  A particularly simple example is Wiener measure of Example \ref{e.3.3}.  Other examples include \cite{cecil_driver,driver_int,driver_gordina_heat,fang,gordina_orthogonal,gordina_wu,malliavin,wu} among others.

Up to now, however, very
little has been said about the associated heat \textit{operator}, which one
would like to realize as an operator on some natural Banach space of
functions. Since there is no infinite-dimensional analog of the
Riemannian volume measure (or Haar measure in the group case), it is natural to use the heat kernel measure
itself. We may attempt, then, to define the heat operator on $L^{p}(M,\mu^o
_{T})$ for some $T>0$ and some fixed basepoint $o\in M$.
Unfortunately, even in the case of an infinite-dimensional Euclidean space, the Laplacian is not a closable operator in $L^{p}(M,\mu^o
_{T}),$ which means that we are not going to be able to construct the heat
operator as any sort of reasonable semigroup mapping $L^{p}(M,\mu^o _{T})$
\textit{to itself}

Our results show, however, that the heat operator $e^{t\Delta /2}$ (or more generally $e^{tL}$) is a nice
operator (a contraction) from $L^{p}(M,\mu^o _{T})$ to a \textit{different
space}, namely $L^{p}(M,\mu^o _{T-t}),$ provided that we
have the natural semigroup property for the heat kernel measures (and this
property will surely hold whenever the measure is constructed as the
transition probabilities of a Markov process). In addition, this approach allows us to define the heat operator on $L^p(M,\mu^o_T)$
directly, without having to work first on some sort of cylinder functions and extend the operator by linearity.

The second infinite-dimensional application is to construct a sort of
\textquotedblleft holomorphic regular version\textquotedblright\ in the
sense of Sugita \cite{sugita_holomorphic,sugita_regular}. In the
infinite-dimensional case, we expect that $\mu^o _{s}$ and $\mu^o _{t}$ will be
mutually singular when $s<t$; the two measures are certainly mutually
singular in the flat case. Nevertheless, in nice cases, holomorphic
functions are automatically harmonic, and so we expect to be able to apply
Theorem \ref{t.1.1} to these functions. This means that the norm of an $L^{p}$
holomorphic function with respect to $\mu^o _{s}$ should be no greater than
the $L^{p}$ norm with respect to $\mu^o _{t},$ even though $\mu^o _{s}$ is
singular with respect to $\mu^o _{t}.$ This monotonicity of the norms is
already known in certain infinite dimensional group cases when $p=2$ and the manifold is a complex group. This follows from the Taylor
expansion, proven first in the context of a complex finite dimensional group by Driver \& Gross \cite{driver_gross}, then extended to infinite dimensional cases by \cite{cecil_taylor,driver_gordina_square,gordina_orthogonal}. Even for general $p$ and a general manifold, the result is
elementary at a formal level, but to our knowledge has not been remarked upon until
now.

It should be noted that in this last application,
we do not seem to be using the fact that the functions involved are
holomorphic, but only that they are harmonic. This suggests that one might
be able to develop a theory of $L^{p}$ harmonic functions parallel to what
one has in the holomorphic case. There are, however, some subtle
difficulties with this idea, connected with the fact that the infinite-dimensional Laplacian is
not a closable operator.

\end{document}